\documentclass[12pt,twoside]{article}

\usepackage{fancyhdr}
\usepackage{amsmath}
\usepackage{amssymb}
\usepackage{amsfonts}
\usepackage{euscript}
\usepackage{oldgerm}
\usepackage{calligra}
\usepackage{mathrsfs}
\usepackage{hyperref}
\usepackage{url}
\usepackage{path}
\usepackage{ecltree}
\usepackage{epsfig}
\usepackage{lastpage}
\usepackage{epic}
\usepackage{multirow}
\usepackage{graphicx}
\usepackage{stmaryrd}
\usepackage{wasysym}
\usepackage{pifont}

\usepackage{draftwatermark}
\SetWatermarkAngle{45}
\SetWatermarkLightness{0.9}
\SetWatermarkFontSize{5cm}
\SetWatermarkScale{2}
\SetWatermarkText{ }

\renewcommand{\thefootnote}{\fnsymbol{footnote}}

\long\def\sfootnote[#1]#2{\begingroup
\def\thefootnote{\fnsymbol{footnote}}\footnote[#1]{#2}\endgroup}

\newtheorem{theorem}{Theorem}
\newtheorem{definition}[theorem]{Definition}

\newtheorem{lemma}[theorem]{Lemma}
\newtheorem{corollary}[theorem]{Corollary}

\newtheorem{example}[theorem]{Example}
\newtheorem{conjecture}[theorem]{Conjecture}
\newtheorem{question}[theorem]{Question}
\newtheorem{problem}[theorem]{Problem}
\newtheorem{remark}[theorem]{Remark}

\newenvironment{proof}{\noindent\mbox{\bf Proof.}}
{\hfill\mbox{$\Subset\!\!\!\!\Supset$}\bigskip}

\begin{document}
\pagestyle{fancy}
\lhead[page \thepage \ (of \pageref{LastPage})]{}
\chead[{\bf Separating Bounded Arithmetics by Herbrand Consistency}]{{\bf Separating Bounded Arithmetics by Herbrand Consistency}}
\rhead[]{page \thepage \ (of \pageref{LastPage})}
\lfoot[\copyright\ {\sf Saeed Salehi 2010}]{$\varoint^{\Sigma\alpha\epsilon\epsilon\partial}_{\Sigma\alpha\ell\epsilon\hslash\imath}\centerdot${\footnotesize {\rm ir}}}
\cfoot[{\footnotesize {\tt http:\!/\!/saeedsalehi.ir/}}]{{\footnotesize {\tt http:\!/\!/saeedsalehi.ir/}}}
\rfoot[$\varoint^{\Sigma\alpha\epsilon\epsilon\partial}_{\Sigma\alpha\ell\epsilon\hslash\imath}\centerdot${\footnotesize {\rm ir}}]{\copyright\ {\sf Saeed Salehi 2010}}
\renewcommand{\headrulewidth}{1pt}
\renewcommand{\footrulewidth}{1pt}
\thispagestyle{empty}

\begin{center}
\begin{table}
\begin{tabular}{| c | l  || l | c |}
\hline
 \multirow{7}{*}{\includegraphics[scale=0.75]{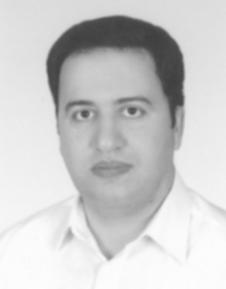}}&    &  &
 \multirow{7}{*}{ \ \ ${\huge \varoint^{\Sigma\alpha\epsilon\epsilon\partial}_{\Sigma\alpha\ell\epsilon\hslash\imath}\centerdot}${{\rm ir}}} \ \ \ \ \\
 &     \ \ {\large{\sc Saeed Salehi}}  \ \  \ & \ \    Tel: \, +98 (0)411 339 2905  \ \  &  \\
 &   \ \ Department of Mathematics \ \  \ & \ \ Fax: \ +98 (0)411 334 2102 \ \  & \\
 &   \ \ University of Tabriz \ \ \  & \ \ E-mail: \!\!{\tt /root}{\sf @}{\tt SaeedSalehi.ir/} \ \  &  \\
 &  \ \ P.O.Box 51666--17766 \ \ \ &   \ \ \ \ {\tt /SalehiPour}{\sf @}{\tt TabrizU.ac.ir/} \ \  &  \\
 &   \ \ Tabriz, Iran \ \ \ & \ \ Web: \  \ {\tt http:\!/\!/SaeedSalehi.ir/} \ \ &  \\
 &    &  &  \\
 \hline
\end{tabular}
\end{table}
\end{center}


\vspace{2em}

\begin{center}
{\bf {\Large  Separating Bounded Arithmetics by Herbrand Consistency
}}
\end{center}

\vspace{2em}
\begin{abstract}
The problem of $\Pi_1-$separating the hierarchy of bounded arithmetic has been studied in the paper. It is shown that the notion of Herbrand Consistency, in its full generality, cannot $\Pi_1-$separate the theory ${\rm I\Delta_0+\bigwedge_j\Omega_j}$ from ${\rm I\Delta_0}$; though it can $\Pi_1-$separate ${\rm I\Delta_0+Exp}$ from ${\rm I\Delta_0}$. This extends a result of L. A. Ko{\l}odziejczyk (2006), by showing the unprovability of the Herbrand Consistency of ${\rm I\Delta_0}$ in the theory ${\rm I\Delta_0+\bigwedge_j\Omega_j}$.

\bigskip

\centerline{${\backsim\!\backsim\!\backsim\!\backsim\!\backsim\!\backsim\!\backsim\!
\backsim\!\backsim\!\backsim\!\backsim\!\backsim\!\backsim\!\backsim\!
\backsim\!\backsim\!\backsim\!\backsim\!\backsim\!\backsim\!\backsim\!
\backsim\!\backsim\!\backsim\!\backsim\!\backsim\!\backsim\!\backsim\!
\backsim\!\backsim\!\backsim\!\backsim\!\backsim\!\backsim\!\backsim\!
\backsim\!\backsim\!\backsim\!\backsim\!\backsim\!\backsim\!\backsim\!
\backsim\!\backsim\!\backsim\!\backsim\!\backsim\!\backsim\!\backsim\!
\backsim\!\backsim\!\backsim\!\backsim\!\backsim\!\backsim\!\backsim\!
\backsim\!\backsim\!\backsim\!\backsim\!\backsim\!\backsim\!\backsim\!
\backsim\!\backsim\!\backsim\!\backsim\!\backsim\!\backsim\!\backsim\!
\backsim\!\backsim\!\backsim}$}

\bigskip



\end{abstract}


\bigskip

\vfill

\begin{center}
\begin{tabular}{| c |}
\hline
  \\
 {\sc Saeed Salehi}, {\bf Separating Bounded Arithmetics by Herbrand Consistency},  \textmd{Manuscript 2010}.  \\
\\
 {\large {\tt http:\!/\!/saeedsalehi.ir/
  }}
     \qquad  \quad
   {\sf Status}: 
 {\sc MANUSCRIPT ({Submitted})} 
  \\
  \\
   \hline
\end{tabular}
\end{center}
\vspace{-1em}
\hspace{.75em} \textsl{\footnotesize Date: 30 August 2010}

\vfill

\bigskip

\centerline{page 1 (of \pageref{LastPage})}


\newpage
\setcounter{page}{2}
\SetWatermarkAngle{55}
\SetWatermarkLightness{0.925}
\SetWatermarkFontSize{50cm}
\SetWatermarkScale{2.25}
\SetWatermarkText{\!\!{\sc MANUSCRIPT (Submitted)}}

\section{Introduction}
One of the consequences of G\"odel's Incompleteness Theorems is the separation of \textsf{Truth} and \textsf{Provability}, in the sense that there are true sentences which are not provable, in sufficiently strong theories. Moreover, those true and unprovable sentences could be $\Pi_1$ (see subsection 3.2). Thus \textsf{Truth} is not $\Pi_1-$conservative over \textsf{Provable}. G\"odel's Second Incompleteness Theorem provides a concrete candidate for $\Pi_1-$separating a theory $T$ over its subtheory $S$, and that is the consistency statement of $S$; when $T$ proves the consistency of $S$, then $T$ is not a $\Pi_1-$conservative extension over $S$, since by the second incompleteness theorem of G\"odel, $S$ cannot prove its own consistency. Indeed, there are lots of $\Pi_1-$separate examples of theories (see subsection 2.2), and there are some difficult open problems relating to $\Pi_1-$separation or $\Pi_1-$conservativeness of arithmetical theories. One of the well-known ones was the $\Pi_1-$separation of ${\rm I\Delta_0+Exp}$, elementary arithmetic, from ${\rm I\Delta_0}$, bounded arithmetic. Here G\"odel's Second Incompleteness Theorem cannot be applied directly, since ${\rm I\Delta_0+Exp}$ does not prove the consistency of ${\rm I\Delta_0}$. For this $\Pi_1-$separation, Paris and Wilkie \cite{PaWi81} suggested the notion of cut-free consistency instead of the usual - Hilbert style - consistency predicate. Here one can show the provability of the cut-free consistency of ${\rm I\Delta_0}$ in the theory ${\rm I\Delta_0+Exp}$, and it was presumed that ${\rm I\Delta_0}$ should not derive its own cut-free consistency (see \cite{Wil07,Sal10} for some historical accounts). But this generalization of G\"odel's Second Incompleteness Theorem, that of unprovability of the weak notions of consistency of weak theories in themselves, took a long time to be established. For example, it was shown that ${\rm I\Delta_0}$ cannot prove the Herbrand Consistency of itself augmented with the axiom of the totality of the squaring function ($\forall x\exists y [y\!=\!x\!\cdot\!x]$) -- see \cite{Wil07}; and then, by a completely different proof, it is shown in \cite{Sal10} the unprovability of the Herbrand Consistency of ${\rm I\Delta_0}$ in itself, when its standard axiomatization is taken. Thus, one line of research was opened for investigating the status of G\"odel's Second Incompleteness Theorem for weak notions of consistencies in weak arithmetics. In another direction, one can ask whether weak notions of consistencies can $\Pi_1-$separate the hierarchies of weak arithmetics. One prominent result here is of L. A.  Ko{\l}odziejczyk \cite{Kol06} in which it was shown that the notion of Herbrand Consistency cannot $\Pi_1-$separate the theory ${\rm I\Delta_0+\bigwedge\Omega_j}$ (see subsection2.2) from ${\rm I\Delta_0+\Omega_1}$. We conjectured in \cite{Sal10} that by using our techniques and methods one can extend this result by showing the unprovability of the Herbrand Consistency of ${\rm I\Delta_0}$  in ${\rm I\Delta_0+\bigwedge\Omega_j}$ (Conjecture 39). In this paper, we prove the cojecture. The arguments of the paper go rather quickly, nevertheless some explanations and examples are presented for clarifying them. No familiarity with the papers cited in the references is assumed for reading this paper; the classic book of Peter H\'ajek and Pavel Pudl\'ak \cite{HP98} is more than enough.
\section{Herbrand Consistency and Bounded Arithmetic}
\subsection{Herbrand Consistency}
For Skolemizing formulas it is  convenient to work with formulas in {\em negation normal form}, which are formulas built up from atomic and negated atomic formulas using $\wedge, \vee, \forall,$ and $\exists$. For having more comfort we consider {\em rectified} formulas, which have the property that different quantifiers refer to different variables, and no variable appears both bound and free. Let us note that any formula can be negation normalized uniquely by converting implication ($A\rightarrow B$) to disjunction ($\neg A\vee B$) and using de Morgan's laws.
And renaming the variables can rectify the formula. Thus any formula can be rewritten in the rectified negation normal form (RNNF) in a somehow unique way (up to a variable renaming).
For any (not necessarily RNNF) existential  formula of the form $\exists x A(x)$, let   ${\mathfrak f}_{\exists x A(x)}$ be a new $m-$ary function symbol where $m$ is the number of the free variables of $\exists x A(x)$.  When $m=0$ then  ${\mathfrak f}_{\exists x A(x)}$ will obviously be a new constant symbol (cf. \cite{Buss95}).
For any RNNF formula $\varphi$ define $\varphi^{\mathsf S}$   by induction:

$\bullet  \ \varphi^{\mathsf S}=\varphi$ for atomic or negated-atomic $\varphi$;

$\bullet \ (\varphi\circ\psi)^{\mathsf S}=\varphi^{\mathsf S}\circ\psi^{\mathsf S}$
for $\circ\!\in\!\{\wedge,\vee\}$ and RNNF formulas $\varphi,\psi$;

$\bullet \ (\forall x\varphi)^{\mathsf S}=\forall x\varphi^{\mathsf S}$;

$\bullet \ (\exists x\varphi)^{\mathsf S}=\varphi^{\mathsf S}[{\mathfrak f}_{\exists x \varphi(x)}(\overline{y})/x]$ where $\overline{y}$ are the free
variables of $\exists x \varphi(x)$ and the formula
$\varphi^{\mathsf S}[{\mathfrak f}_{\exists x \varphi(x)}(\overline{y})/x]$
results from the formula $\varphi^{\mathsf S}$ by replacing all the occurrences of
the variable $x$ with the term ${\mathfrak f}_{\exists x \varphi(x)}(\overline{y})$.

\noindent Finally the Skolemized form of a formula $\psi$ is obtained by

(1) negation normalizing and rectifying it to $\varphi$;

(2) getting $\varphi^{\mathsf S}$ by the above inductive procedure;

(3)  removing all the remaining  (universal) quantifiers  in $\varphi^{\mathsf S}$.

\noindent We denote thus resulted Skolemized form of $\psi$ by $\psi^{\rm Sk}$. Note that our way of Skolemizing did not need prenex normalizing a formula. And it results in a unique (up to a variable renaming) Skolemized formula.
\begin{example}\label{example1}
{\rm Take $0$ be a constant symbol, ${\mathfrak s}$ be a unary function symbol, $+$ and $\cdot$ be two binary function symbols, and $\leqslant$ be a binary predicate symbol. Let $A$ be the sentence $\forall x\forall y (x\leqslant y \leftrightarrow \exists z
[z+x=y])$ which is an axiom of Robinson's Arithmetic ${\rm Q}$ (see Example \ref{example2}), and let $B$ be  $\theta(0)\!\wedge\!\forall x [\theta(x)\!\rightarrow\!\theta(x+1)]\!\Rightarrow\!\forall x\theta(x)$ where  $\theta(x)\!=\!\exists y(y\!\leqslant\!x\!\cdot\!x\!\wedge\!y\!=\!x\!\cdot\!x)$. This is an axiom of the theory ${\rm I\Delta_0}$ (see subsection 2.2). The rectified negation normalized forms of these sentences can be obtained as follows:

$C=A^{\rm RNNF}=\forall x\forall y \big([x\not\!\leqslant\!y\vee \exists u (u+x=y)] \wedge [\forall z (z+x\not=y)\vee x\!\leqslant\!y]\big)$, and

$D=B^{\rm RNNF}=\forall u (u\not\!\leqslant\!0\!\cdot\!0\vee u\not=0\!\cdot\!0)\bigvee$

\qquad \ \qquad \ \qquad   $\exists w\big[(\exists z [z\!\leqslant\!w\!\cdot\!w\wedge z=w\!\cdot\!w])\wedge (\forall v[v\not\!\leqslant\!({\mathfrak s}w)\!\cdot\!({\mathfrak s}w)\vee v\not=({\mathfrak s}w)\!\cdot\!({\mathfrak s}w)])\big]\bigvee$

\qquad \ \qquad \ \qquad   $\forall x\exists y [y\!\leqslant\!x\!\cdot\!x\wedge y=x\!\cdot\!x]$.

\noindent Let ${\mathfrak h}$ stand for ${\mathfrak f}_{\exists u
(u+x=y)}$, ${\mathfrak q}(\xi)$ be the Skolem function symbol for the formula $\exists z [z\!\leqslant\!\xi\!\cdot\!\xi\!\wedge\!z\!=\!\xi\!\cdot\!\xi]$, and ${\mathfrak c}$ abbreviate the Skolem constant symbol for  $\exists w\big[(\exists z [z\!\leqslant\!w\!\cdot\!w\wedge z=w\!\cdot\!w])\wedge (\forall v[v\not\!\leqslant\!({\mathfrak s}w)\!\cdot\!({\mathfrak s}w)\wedge v\not=({\mathfrak s}w)\!\cdot\!({\mathfrak s}w)])\big]$.
Then $C^{\mathsf S}$ and $D^{\mathsf S}$ are:

$C^{\mathsf S}=\forall x\forall y \big([x\not\!\leqslant\!y\vee  ({\mathfrak h}(x,y)+x=y)] \wedge [\forall z (z+x\not=y)\vee x\!\leqslant\!y]\big)$, and

$D^{\mathsf S}=\forall u (u\not\!\leqslant\!0\!\cdot\!0\vee u\not=0\!\cdot\!0)\bigvee$

 \quad \ \  \quad $\big[({\mathfrak q}({\mathfrak c})\!\leqslant\!{\mathfrak c}\!\cdot\!{\mathfrak c}\wedge {\mathfrak q}({\mathfrak c})={\mathfrak c}\!\cdot\!{\mathfrak c})\wedge \forall v(v\not\!\leqslant\!({\mathfrak s}{\mathfrak c})\!\cdot\!({\mathfrak s}{\mathfrak c})\vee v\not=({\mathfrak s}{\mathfrak c})\!\cdot\!({\mathfrak s}{\mathfrak c}))\big]\bigvee$

\quad \ \  \quad $\forall x ({\mathfrak q}(x)\!\leqslant\!x\!\cdot\!x\wedge {\mathfrak q}(x)=x\!\cdot\!x)$.

\noindent Finally the Skolemized forms of $A$ and $B$ are obtained as:

$A^{\rm Sk}=[x\not\!\leqslant\!y\vee  ({\mathfrak h}(x,y)+x=y)] \wedge [(z+x\not=y)\vee x\!\leqslant\!y]$, and

$B^{\rm Sk}= (u\not\!\leqslant\!0\!\cdot\!0\vee u\not=0\!\cdot\!0)\bigvee$

 \qquad   \quad  $\big[({\mathfrak q}({\mathfrak c})\!\leqslant\!{\mathfrak c}\!\cdot\!{\mathfrak c}\wedge {\mathfrak q}({\mathfrak c})={\mathfrak c}\!\cdot\!{\mathfrak c})\wedge (v\not\!\leqslant\!({\mathfrak s}{\mathfrak c})\!\cdot\!({\mathfrak s}{\mathfrak c})\vee v\not=({\mathfrak s}{\mathfrak c})\!\cdot\!({\mathfrak s}{\mathfrak c}))\big]\bigvee$

 \qquad   \quad $({\mathfrak q}(x)\!\leqslant\!x\!\cdot\!x\wedge {\mathfrak q}(x)=x\!\cdot\!x)$.
 \hfill$\subset\!\!\!\!\supset$}\end{example}
An {\em Skolem instance} of a formula $\psi$ is any formula resulted
from substituting the free variables of $\psi^{\rm Sk}$ with some
terms. So, if $x_1,\ldots,x_n$ are the free variables of
$\psi^{\rm Sk}$ (thus written as $\psi^{\rm
Sk}(x_1,\ldots,x_n)$) then an Skolem instance of $\psi$ is
$\psi^{\rm Sk}[t_1/x_1,\cdots,t_n/x_n]$ where $t_1,\ldots,t_n$
are terms (which could be constructed from the Skolem functions
symbols).
 Skolemized form of a theory
$T$ is by definition $T^{\rm Sk}=\{\varphi^{\rm Sk}\mid\varphi\!\in\! T\}$. Herbrand's Theorem appears in several forms in the literature. As we wish to arithmetize a somewhat general notion of Herbrand Consistency, below we present a version of Herbrand's fundamental theorem, also attributed to G\"odel and Skolem, which will make the formalization easier  (cf. \cite{Buss95}).
\begin{theorem}[G\"odel - Herbrand - Skolem]{\rm Any theory $T$ is
equiconsistent with its Skolemized theory $T^{\rm Sk}$. Or in other words, $T$ is consistent if and only if every finite set of Skolem
instances of $T$ is (propositionally) satisfiable.\hfill $\Subset\!\!\!\!\Supset$
}\end{theorem}
Our means of propositional satisfiability is by {\em evaluations}, which are defined to be any function $p$ whose domains are the set of all
atomic formulas constructed from a given set of terms $\Lambda$ and
whose  ranges are the set $\{0,1\}$ such that

 (1) \ $p \, [t\!=\!t]=1$ for all $t\!\in\!\Lambda$; and for any terms
$t,s\!\in\!\Lambda$,

 (2) \  if $p\, [t\!=\!s]=1$ then $p\, [\varphi(t)]=p\,
[\varphi(s)]$ for any atomic formula $\varphi(x)$.

\noindent The relation $\backsim_p$ on $\Lambda$ is defined by
$t\backsim_p s \iff p[t=s]=1$ for $t,s\!\in\!\Lambda$.
One can see that the  relation
$\backsim_p$   is an equivalence relation, and moreover is a  congruence relation as well. That is, for any set of terms $t_i$ and $s_i$ ($i=1,\ldots,n$) and function symbol $f$, if $p[t_1=s_1]=\cdots p[t_n=s_n]=1$ then $p[f(t_1,\ldots,t_n)=f(s_1,\ldots,s_n)]=1$.

\noindent  The $\backsim_p\!\!-$class of a term
$t$ is denoted by $t/p$; and the set of all such $p-$classes for
each $t\!\in\!\Lambda$ is denoted by $\Lambda/p$.
For simplicity, we write $p\models\varphi$ instead of
$p\,[\varphi]=1$; thus $p\not\models\varphi$ stands for
$p\,[\varphi]=0$. This definition of {\em satisfying} can be
generalized to other open (RNNF) formulas as usual.

 If all terms appearing in an Skolem instance of $\varphi$
belong to the set $\Lambda$, that formula is called an  Skolem
instance of $\varphi$ {\em available} in $\Lambda$.
 An evaluation defined on $\Lambda$ is called a {\em
$\varphi-$evaluation} if it satisfies all the Skolem instances of
$\varphi$ which are available in $\Lambda$.
 Similarly, for a theory $T$,  a {\em $T-$evaluation} on
$\Lambda$ is an evaluation on $\Lambda$ which satisfies every Skolem
instance of every formula of $T$ which is available in
$\Lambda$.
 By Herbrand's Theorem, a theory $T$ is consistent if and only if for every set of terms $\Lambda$ (constructed from the Skolem terms of axioms of $T$) there exists a $T-$evaluation on $\Lambda$. We will use this reading of Herbrand's Theorem for defining the notion of {\em Herbrand Consistency}. Thus {\em Herbrand Provability} of a formula $\varphi$ in a theory $T$ is equivalent to the existence of a set of terms  on which there cannot exist any $(T\cup\{\neg\varphi\})-$evaluation.
\begin{example}\label{example2}
{\rm Let ${\rm Q}$ denote Robinson's Arithmetic over
the language of arithmetic $\langle0,{\mathfrak s},+,\cdot,\leqslant\rangle$, where $0$ is a
constant symbol, ${\mathfrak s}$ is a unary function symbol, $+,\cdot$ are binary
function symbols, and $\leqslant$ is a binary predicate symbol, whose axioms are:
\begin{align*}
A_1: & \ \  \forall x ({\mathfrak s} x\not=0) & A_2: & \ \  \forall x\forall y ({\mathfrak s} x ={\mathfrak s} y\rightarrow x=y)\\
A_3: & \ \   \forall x (x\not=0\rightarrow \exists y[x={\mathfrak s} y]) & \ \   A_4: & \ \   \forall x\forall y (x\!\leqslant\!y \leftrightarrow \exists z
[z+x=y])\\
A_5: & \ \   \forall x (x+0=x) & A_6: & \ \   \forall x\forall y (x+{\mathfrak s} y={\mathfrak s} (x+y))\\
A_7:  & \ \   \forall x (x\cdot 0=0) & A_8: & \ \   \forall x\forall y (x\cdot{\mathfrak s} y=x\cdot y + x)
\end{align*}
Let $\varphi=\forall x(x\!\leqslant\!0\!\rightarrow\!x\!=\!0)$. We can show
${\rm Q}\vdash\varphi$; this will be proved below by Herbrand Provability.
Suppose ${\rm Q}$ has been Skolemized as below:
\begin{align*}
A_1^{\rm Sk}: & \ \   {\mathfrak s} x\not=0 & A_2^{\rm Sk}: & \ \   {\mathfrak s} x\not={\mathfrak s} y\vee x=y\\
A_3^{\rm Sk}: & \ \   x=0 \vee x={\mathfrak s}{\mathfrak p} x & A_4^{\rm Sk}: & \ \   [x\not\!\leqslant\!y\vee {\mathfrak h}(x,y)+x=y]\wedge[z+x\not=y\vee x\!\leqslant\!y]\\
A_5^{\rm Sk}: & \ \   x+0=x & A_6^{\rm Sk}: & \ \   x+{\mathfrak s} y={\mathfrak s} (x+y)\\
A_7^{\rm Sk}:  & \ \   x\cdot 0=0 & A_8^{\rm Sk}: & \ \   x\cdot{\mathfrak s} y=x\cdot y + x
\end{align*}
Here ${\mathfrak p}$ abbreviates ${\mathfrak f}_{\exists
y(x={{\mathfrak s} }y)}$ and ${\mathfrak h}$ stands for ${\mathfrak f}_{\exists z
(z+x=y)}$. Suppose $\neg\varphi$ has been Skolemized as $({\mathfrak c}\!\leqslant\!0\wedge{\mathfrak c}\not=0)$ where ${\mathfrak c}$ is the Skolem constant symbol for  $\exists x (x\!\leqslant\!0\wedge x\not=0)$. Take $\Lambda$ be the following set of terms $\Lambda=\{0, {\mathfrak c}, {\mathfrak h}({\mathfrak c},0), {\mathfrak h}({\mathfrak c},0)+{\mathfrak c}, {\mathfrak s}{\mathfrak p}{\mathfrak c}, {\mathfrak s}({\mathfrak h}({\mathfrak c},0)+{\mathfrak p}{\mathfrak c}), {\mathfrak h}({\mathfrak c},0)+{\mathfrak s}{\mathfrak p}{\mathfrak c}\}$. We show that there is no $({\rm Q}+\neg\varphi)-$evaluation on $\Lambda$. Assume (for the sake of contradiction) that $p$ is such an evaluation. Then by $A_3$ we have $p\models {\mathfrak c}={\mathfrak s}{\mathfrak p}{\mathfrak c}$. On the other hand by $A_4$ we have $p\models {\mathfrak h}({\mathfrak c},0)+{\mathfrak c}=0$, and so $p\models {\mathfrak h}({\mathfrak c},0)+{\mathfrak s}{\mathfrak p}{\mathfrak c}=0$. Then by $A_6$ we get $p\models {\mathfrak s}({\mathfrak h}({\mathfrak c},0)+{\mathfrak p}{\mathfrak c})=0$ which is a contradiction with $A_1$.
 \hfill$\subset\!\!\!\!\supset$
}\end{example}
Let us note that finding a suitable set of terms $\Lambda$ for which there cannot exist a $(T+\neg\psi)-$evaluation on $\Lambda$ is as complicated as finding a proof of $T\vdash\psi$ (even more complicated - see subsection 3.1).
%
The following is another example for illustrating the concepts of Skolem instances and evaluations, which will be used later in the paper (the proof of Theorem \ref{th:b}).
\begin{example}\label{q}
{\rm Let $B$ be as in the Example \ref{example1}, in the language $\langle0,{\mathfrak s},+,\cdot,\leqslant\rangle$. Thus,
\newline\centerline{$B=\theta(0)\!\wedge\!\forall x[\theta(x)\!\rightarrow\!\theta({\mathfrak s}x)]\!\rightarrow\!\forall x\theta(x)$
\ \ where \ \  $\theta(x)=\exists y(y\!\leqslant\!x\!\cdot\!x\!\wedge\!y\!=\!x\!\cdot\!x)$.}
We saw that the Skolemized form of $B$ is

$B^{\rm Sk}= (u\not\!\leqslant\!0\!\cdot\!0\vee u\not=0\!\cdot\!0)\bigvee$

 \qquad   \quad  $\big[({\mathfrak q}({\mathfrak c})\!\leqslant\!{\mathfrak c}\!\cdot\!{\mathfrak c}\wedge {\mathfrak q}({\mathfrak c})={\mathfrak c}\!\cdot\!{\mathfrak c})\wedge (v\not\!\leqslant\!({\mathfrak s}{\mathfrak c})\!\cdot\!({\mathfrak s}{\mathfrak c})\vee v\not=({\mathfrak s}{\mathfrak c})\!\cdot\!({\mathfrak s}{\mathfrak c}))\big]\bigvee$

 \qquad   \quad $({\mathfrak q}(x)\!\leqslant\!x\!\cdot\!x\wedge {\mathfrak q}(x)=x\!\cdot\!x)$,

\noindent where ${\mathfrak q}(\xi)$ is the Skolem function symbol for the formula $\exists z [z\!\leqslant\!\xi\!\cdot\!\xi\!\wedge\!z\!=\!\xi\!\cdot\!\xi]$ and ${\mathfrak c}$ is the Skolem constant of {$\exists w\big[(\exists z [z\!\leqslant\!w\!\cdot\!w\wedge z=w\!\cdot\!w])\wedge (\forall v[v\not\!\leqslant\!({\mathfrak s}w)\!\cdot\!({\mathfrak s}w)\wedge v\not=({\mathfrak s}w)\!\cdot\!({\mathfrak s}w)])\big]$.}
Define the set of terms $\Upsilon$ by {$\Upsilon=\{0,  0+0,  0^2, {\mathfrak c}, {\mathfrak c}^2, {\mathfrak c}^2+0, {\mathfrak s}{\mathfrak c}, {\mathfrak q}{\mathfrak c},  ({\mathfrak s}{\mathfrak c})^2,  ({\mathfrak s}{\mathfrak c})^2+0\}$} and suppose $p$ is an $(Q+B)-$evaluation on the set of terms $\Upsilon\cup\{t,t^2,{\mathfrak q}(t)\}$. The notation $\varrho^2$ is a shorthand for $\varrho\cdot\varrho$. Then $p$ must satisfy the following Skolem instance  of $B$ which is available in the set $\Upsilon\cup\{t,t^2,{\mathfrak q}(t)\}$:

$(\eth) \ \ \ \  (0\not\leqslant 0^2\vee 0\not=0^2)  \bigvee$

\qquad \; $\Big(\big({\mathfrak q}{\mathfrak c}\leqslant{\mathfrak c}^2\wedge{\mathfrak q}{\mathfrak c}={\mathfrak c}^2\big)\wedge\big(({\mathfrak s}{\mathfrak c})^2\not\leqslant({\mathfrak s}{\mathfrak c})^2
\vee ({\mathfrak s}{\mathfrak c})^2\not=({\mathfrak s}{\mathfrak c})^2\big)\Big)  \bigvee$

\qquad \; $\big({\mathfrak q}(t)\leqslant t^2\wedge {\mathfrak q}(t)=t^2\big).$

\noindent Now since $p\models 0\cdot 0=0+0=0$ then, by ${\rm Q}$'s axioms, $p\models 0\leqslant 0^2\wedge 0=0^2$, and so $p$ cannot satisfy the first disjunct of $(\eth)$. Similarly, since $p\models ({\mathfrak s}{\mathfrak c})^2+0=({\mathfrak s}{\mathfrak c})^2$ then $p\models ({\mathfrak s}{\mathfrak c})^2\leqslant ({\mathfrak s}{\mathfrak c})^2$, thus $p$ cannot satisfy the second disjunct of $(\eth)$ either, because $p\models ({\mathfrak s}{\mathfrak c})^2=({\mathfrak s}{\mathfrak c})^2$. Whence,  $p$ must satisfy the third disjunct of $(\eth)$, then necessarily $p\models {\mathfrak q}(t)=t^2$ must hold.
\hfill $\subset\!\!\!\!\supset$
}\end{example}
\subsection{Bounded Arithmetic Hierarchy}
First-order Peano arithmetic ${\rm PA}$ is the theory in the language $\langle0,{\mathfrak s},+,\cdot,\leqslant\rangle$ axiomatized by Robinson's Arithmetic ${\rm Q}$ (see Example \ref{example2}) plus the induction schema $\psi(0)\wedge\forall x[\psi(x)\rightarrow\psi({\mathfrak s}(x))]\Rightarrow\forall x\psi(x)$ for any formula $\psi(x)$. This theory is believed to encompass a large body of arithmetical truth in mathematics; the most recent conjecture (due to H. Friedman) is that a proof of Fermat's Last Theorem can be carried out inside ${\rm PA}$ (\cite{Avi03}), and indeed Andrew Wiles's proof of the theorem has been claimed to be formalized in it (\cite{Mac}). To see a simpler example, we note that
primality can be expressed in the language of arithmetic by the following formula:
$\textsf{Prime}(x)\equiv \forall y,z (y\cdot z=x\rightarrow y=1\vee z=1)$.
Then Euclid's theorem on the infinitude of the primes can be written as $\forall x \exists y [y>x \wedge \textsf{Prime}(y)]$. It can be shown that Euclid's proof can be formalized completely in ${\rm PA}$. One would wish to see how much strength of ${\rm PA}$ is necessary for proving the infinitude of the primes. An important sub-theory of Peano's Arithmetic is introduced by R. Parikh (\cite{Parikh}) as follows. A formula is called bounded  if its every quantifier is bounded, i.e., is either of the
form $\forall x\!\leqslant\! t(\ldots)$ or $\exists\, x\!\leqslant\! t(\ldots)$ where $t$ is a term; they are read as $\forall x (x\!\leqslant\! t\rightarrow\ldots)$ and $\exists x (x\!\leqslant\! t \wedge\ldots)$ respectively. The class of bounded formulas is denoted by  ${\rm \Delta_0}$. It is easy to see that bounded formulas are decidable. The theory ${\rm I\Delta_0}$, also called bounded arithmetic, is axiomatized by ${\rm Q}$ plus the induction schema for bounded formulas. An important property of this arithmetic  is that whenever ${\rm I\Delta_0}\vdash \forall \overline{x}\exists y  \ \theta(\overline{x})$ for a bounded formula $\theta$, then there exists a term (polynomial) $t(\overline{x})$ such that ${\rm I\Delta_0}\vdash \forall \overline{x}\exists y\!\leqslant\!t(\overline{x})  \ \theta(\overline{x})$ (see e.g. \cite{HP98}). An open problem in the theory of weak arithmetics is that whether or not the infinitude  of the primes can be proved inside ${\rm I\Delta_0}$. However, it is known that much of elementary number theory cannot be proved inside ${\rm I\Delta_0}$; the theory is too weak to even recognize the totality of the exponentiation function. The exponentiation function $\exp$ is defined by $\exp(x)=2^x$; the formula ${\rm Exp}$ expresses its totality: ($\forall x\exists y[y\!=\!\exp(x)]$). The theory ${\rm I\Delta_0+Exp}$, sometimes called Elementary Arithmetic, is able to formalize much of number theory. It can surely prove the infinitude  of the primes. Note that in Euclid's proof, for getting a prime number greater than $x$ one can use $x!+1$ which should have a prime factor greater than $x$ (no number non-greater than $x$ can divide it). And it can be seen that $x!<\exp\exp(x)$. Between ${\rm  I\Delta_0}$ and ${\rm I\Delta_0+Exp}$ a hierarchy of theories is considered in the literature, which has close connections with computational complexity. They are sometimes called weak arithmetics, and sometimes bounded arithmetics. The hierarchy is defined below.
The converse of $\exp$ is denoted by $\log$ which is formally defined as $\log x={\rm min}\{y\mid x\!\leqslant\!\exp(y)\}$; thus $\exp(\log x -1)\!<\!x\!\leqslant\!\exp(\log x)$. The superscripts above the function symbols indicate the iteration of the functions; e.g.,  $\exp^2(x)=\exp\exp(x)$ and $\log^3 x=\log\log\log x$. Define the function $\omega_m$ to be $\omega_m(x)=\exp^m\big((\log^m x)\cdot(\log^m x)\big)$. It is customary to define this function by induction:  ${\rm \omega_0}(x)=x^2$ and ${\rm \omega_{n+1}}(x)=\exp({\rm \omega_n}(\log x))$. Let  ${\rm \Omega_m}$ express the totality of ${\rm \omega_m}$ (i.e., ${\rm \Omega_m}\equiv \forall x\exists y [y={\rm \omega_m}(x)]$). It can be more convenient to consider the function $\omega_{-1}(x)=2x$  as well (cf. \cite{Kol06}). The hierarchy between ${\rm I\Delta_0}$ and ${\rm I\Delta_0+Exp}$ is $\{{\rm I\Delta_0+\Omega_m}\}_{{\rm m}\!\geqslant\!1}$. For example,  the theory ${\rm I\Delta_0+\Omega_1}$ can prove the infinitude  of the primes (the proof is not easy at all - see \cite{PWW}).   We first review some basic properties of the $\omega_n$ functions: $\omega_1$ dominates all the polynomials and $\omega_{m+1}$ dominates all the (finite) iterations of $\omega_m$. Let us note that $\omega_0^N(x)=x^{\exp(N)}$ and $\omega_m^N(x)=\exp^m([\log^m x]^{\exp(N)})$, also $\omega_{j+1}^N(x)=\exp(\omega_j^N(x))$, hold for any $N\!\geqslant\!1$.
\begin{lemma}\label{lem1}
{\rm For any natural $m\!\geqslant\!0$ and $N\!>\!2$, and any $x\!>\!\exp^{m+2}(N)$, we have $\omega_m^N(x)\!<\!\omega_{m+1}(x)$.
}\end{lemma}

\begin{proof}
For $m=0$ we note that $2^N\!\cdot\!\log x\!<\!(\log x)^2$ for any $x\!>\!\exp^2(N)$. Thus $\exp (2^N\log x)\!<\!\exp((\log x)^2)$, which implies that $\omega_0^N(x)\!<\!\omega_1(x)$.

For $m\!\geqslant\!1$ we can use an inductive argument. For any $x\!>\!\exp^{m+2}(N)$ we have $\log x\!>\!\exp^{m+1}(N)$,  so by the induction hypothesis $\omega_{m-1}^N(\log x)\!<\!\omega_m(\log x)$ holds. Then $\exp[\omega_{m-1}^N(\log x)]\!<\!\exp[\omega_m(\log x)]$, and so $\omega_m^N(x)\!<\!\omega_{m+1}(x)$.
\end{proof}

We now present a generalization of the lemma, which will be used later.
\begin{lemma}\label{lem2}
{\rm For any  $m\!\geqslant\!-1, N\!\geqslant\!1$ and $x\!>\!\exp^{m+2}(4N+4)$, there exists some $y  \ (\leqslant\!x)$ such that $$\omega_m^N(y)\!<\!x\!\leqslant\!\omega_{m+1}(y).$$
}\end{lemma}

\begin{proof} We first show the lemma for $m=-1$: for any $x\!>\!\exp(4N+4)$, there exists a least $y$ such that $y^2\!\geqslant\!x$; so $(y-1)^2\!<\!x$. Also from $y^2\!>\!2^{{4N+4}}$ we have $y\!>\!2^{2N+2}$. Whence we have $x\!\leqslant\!y^2=\omega_0(y)$, and also $\omega_{-1}^N(y)=2^N\cdot y\!<\!\sqrt{y}\cdot y\!\leqslant\!(y-1)^2\!<\!x$. Let us note that  $\sqrt{y}\cdot y\!\leqslant\!(y-1)^2$ holds for any $y\!\geqslant\!4$ and we have $y\!>\!2^{2N+2}\!>\!4$.

For $m=0$, we use the above argument for $\log x$, noting that $\log x\!>\!\exp(4N+4)$ holds by the assumption $x\!>\!\exp^2(4N+4)$. There must exist some $z$ such that $2^N\cdot z\!<\!\log x\!\leqslant\!z^2$. Let $y=\exp(z)$, so $z=\log y$. Thus from $2^N\log y\!<\!\log x\!\leqslant\!(\log y)^2$ it follows that $\omega_0^N(y)=y^{\exp(N)}\!\leqslant\!\exp[\exp(N)\cdot(\log y)]\!\leqslant\!\exp(\log x -1)\!<\!x\!\leqslant\!\exp(\log x)\!\leqslant\!\exp([\log y]^2)=\omega_1(y)$.

For $m\!\geqslant\!1$, we can use induction on $m$ with a  straightforward argument. For $x\!>\!\exp^{m+3}(4N+4)$, we have $\log x\!>\!\exp^{m+2}(4N+4)$, and so by the induction hypothesis there exists a $z$ such that the inequalities $\omega_m^N(z)\!<\!\log x\!\leqslant\!\omega_{m+1}(z)$ hold. Put $y=\exp(z)$, so we have $\omega_m^N(\log y)\!<\!\log x\!\leqslant\!\omega_{m+1}(\log y)$. Thus,

$\omega_{m+1}^N(y)=\exp(\omega_m^N(\log y))\!\leqslant\!\exp(\log x-1)\!<\!x\!\leqslant\!\exp(\log x)\!\leqslant\!\exp(\omega_{m+1}(\log y))=
\omega_{m+2}(y)$.
\end{proof}

Whence the hierarchy $\{{\rm I\Delta_0+\Omega_m}\}_{{\rm m}\!\geqslant\!1}$ is proper; in the sense that $$\circledast  \ \ \ \ \ \ \ {\rm I\Delta_0}\subsetneqq{\rm I\Delta_0+\Omega_1}\subsetneqq\cdots{\rm I\Delta_0+\Omega_n}\subsetneqq{\rm I\Delta_0+\Omega_{n+1}}\subsetneqq\cdots\subsetneqq{\rm I\Delta_0+\bigwedge\Omega_j}\subsetneqq{\rm I\Delta_0+Exp}.$$ The notation ${\rm I\Delta_0+\bigwedge\Omega_j}$ abbreviates $\bigcup_{\rm n\geqslant1}({\rm I\Delta_0+\Omega_n})$.
The class of $\Sigma_n-$formulas and $\Pi_n-$formulas are defined as follows:
$\Sigma_1-$formulas are equivalently in the form $\exists \overline{x}\theta(\overline{x})$, where $\theta\!\!\in\!\!{\rm \Delta_0}$, and $\Pi_1-$formulas are equivalently in the form $\forall \overline{x}\theta(\overline{x})$, for some $\theta\!\!\in\!\!{\rm \Delta_0}$. Then  $\Sigma_{n+1}-$formulas are equivalent to $\exists \overline{x}\varphi(\overline{x})$ for some $\varphi\!\!\in\!\!\Pi_n$, and $\Pi_{n+1}-$formulas are equivalent to $\forall \overline{x}\varphi(\overline{x})$ for some $\varphi\!\!\in\!\!\Sigma_n$. The above hierarchy is not $\Pi_2-$conservative, i.e., there exists a $\Pi_2-$formula (namely $\Omega_{m+1}$) which is provable in ${\rm I\Delta_0+\Omega_{m+1}}$ but not in ${\rm I\Delta_0+\Omega_{m}}$. Though, the (difficult) open problem here is the $\Pi_1-$conservativity of the hierarchy:
\begin{problem}
{\rm Is there a ${\rm\Pi_1}-$sentence $\psi$ such that ${\rm I\Delta_0+\Omega_{m+1}}\vdash\psi$ and ${\rm I\Delta_0+\Omega_{m}}\not\vdash\psi$?
\hfill $\subset\!\!\!\!\supset$
}\end{problem}
As for the above hierarchy $\circledast$ it is (only) known that ${\rm I\Delta_0+Exp}$ is not $\Pi_1-$conservative over ${\rm I\Delta_0+\bigwedge\Omega_j}$ (see \cite{HP98}, Corollary 5.34 and the afterward explanation).

Examples of $\Pi_1-$separation   abound in mathematics and logic: Zermelo-Frankel Set Theory ${\rm ZFC}$ is not $\Pi_1-$conservative over Peano's Arithmetic ${\rm PA}$, because we have ${\rm ZFC}\vdash {\rm Con}({\rm PA})$ but, by G\"odel's Second Incompleteness Theorem,
${\rm PA}\not\vdash {\rm Con}({\rm PA})$; where ${\rm Con}(-)$ is the consistency predicate. Inside ${\rm PA}$ the $\Sigma_n-$hierarchy is not a $\Pi_1-$conservative hierarchy, since ${\rm I\Sigma_{n+1}}\vdash{\rm Con}({\rm I\Sigma_n})$ though ${\rm I\Sigma_{n}}\not\vdash{\rm Con}({\rm I\Sigma_n})$; see e.g. \cite{HP98}. Then below the  theory ${\rm I\Sigma_1}$ things get more complicated: for $\Pi_1-$separating ${\rm I\Delta_0}+{\rm Exp}$ over ${\rm I\Delta_0}$ the candidate ${\rm Con}({\rm I\Delta_0})$ does not work as expected, because ${\rm I\Delta_0}+{\rm Exp}\not\vdash{\rm Con}({\rm I\Delta_0})$ (see \cite{HP98} Corollary~5.29). For this $\Pi_1-$separation, Paris and Wilkie \cite{PaWi81} suggested the notion of cut-free consistency instead of the usual - Hilbert style - consistency predicate. Here one can show that ${\rm I\Delta_0}+{\rm Exp}\vdash{\rm CFCon}({\rm I\Delta_0})$, and then it was presumed that ${\rm I\Delta_0}\not\vdash {\rm CFCon}({\rm I\Delta_0})$, where ${\rm CFCon}$ stands for cut-free consistency. In 2006, L. A. Ko\l odziejczyk \cite{Kol06} showed that the notion of Herbrand Consistency (and thus, more probably, other Cut-Free consistencies, like Tableaux etc.) will not work for $\Pi_1-$separating the hierarchy above ${\rm I\Delta_0+\Omega_1}$ either. Namely, ${\rm I\Delta_0+\bigwedge\Omega_j}\not\vdash{\rm HCon}({\rm I\Delta_0+\Omega_1})$, where ${\rm HCon}(-)$ is the predicate of Herbrand Consistency (see subsection 2.3). In this paper, we extend this rather negative result one step further, by proving ${\rm I\Delta_0+\bigwedge\Omega_j}\not\vdash{\rm HCon}({\rm I\Delta_0})$.
\subsection{Herbrand Consistency in Bounded Arithmetics}
For a theory $T$, when $\Lambda$ is the set of all terms
(constructed from the function symbols of the language of $T$ and
also the Skolem function symbols of the formulas of $T$) any
$T-$evaluation on $\Lambda$ induces a model of $T$, which is called
a {\em Herbrand Model}.
Let ${\mathcal L}$ be a language  and $\Lambda$ be a
 set of (ground) terms (constructed by the Skolem constant and function symbols of ${\mathcal L}$).

 Put $\Lambda^{\langle 0\rangle}=\Lambda$, and define inductively

$\Lambda^{\langle k+1\rangle}=\Lambda^{\langle k\rangle}\cup
\{f(t_1,\ldots,t_m)\mid f\!\in\!{\mathcal L}
\;\&\;t_1,\ldots,t_m\!\in\!\Lambda^{\langle k\rangle}\}$

\hspace{5.9em}  $\cup  \,\{{\mathfrak f}_{\exists
x\psi(x)}(t_1,\ldots,t_m)\mid \ulcorner\psi\urcorner\leqslant k \;\&\;
t_1,\ldots,t_m\!\in\!\Lambda^{\langle k\rangle}\}$.

 Let $\Lambda^{\langle \infty\rangle}$ denote the union
$\bigcup_{k\in\mathbb{N}}\Lambda^{\langle k\rangle}$.

\noindent Suppose $p$ is an evaluation on $\Lambda^{\langle \infty\rangle}$.
Define ${\textswab M}(\Lambda,p)=\{t/p\mid t\!\in\!\Lambda^{\langle
\infty\rangle}\}$ and put the ${\mathcal L}-$structure on it by

$(1) \ f^{{\textswab M}(\Lambda,p)}(t_1/p,\ldots,t_m/p)=f(t_1,\ldots,t_m)/p$, and

$(2) \ R^{{\textswab M}(\Lambda,p)}=\{(t_1/p,\ldots,t_m/p)\mid p\models
R(t_1,\ldots,t_m)\}$,

\noindent for function symbol $f$, relation symbols $R$, and terms
$t_1,\ldots,t_m\!\in\!\Lambda^{\langle\infty\rangle}$.
\begin{lemma}\label{lm:mlambdap}
{\rm The definition of ${\mathcal L}-$structure on ${\textswab M}(\Lambda,p)$ is
well-defined, and when $p$ is an $T-$evaluation on
$\Lambda^{\langle\infty\rangle}$, for an ${\mathcal L}-$theory $T$, then
${\textswab M}(\Lambda,p)\models T$. }\end{lemma}

\begin{proof} That the definitions of  $f^{{\textswab M}(\Lambda,p)}$ and $R^{{\textswab M}(\Lambda,p)}$ are well-defined follows directly from  the definition of an evaluation. By the definition of $\Lambda^{\langle\infty\rangle}$ the structure ${\textswab M}(\Lambda,p)$ is closed under all the Skolem functions of ${\mathcal L}$, and moreover it satisfies an atomic (or negated atomic) formula
$A(t_1/p,\ldots,t_m/p)$ if and only if $p\models A(t_1,\ldots,t_m)$.
Then it can be shown, by induction on the complexity of formulas,
that for every RNNF formula $\psi$, we have ${\textswab M}(\Lambda,p)\models\psi$ whenever $p$ satisfies all the available
Skolem instances of $\psi$ in $\Lambda^{\langle\infty\rangle}$. Whence, if $p$ is a $T-$evaluation, then we have ${\textswab M}(\Lambda,p)\models T$.
\end{proof}

For arithmetizing the notion of Herbrand Consistency, we adopt an efficient G\"odel coding, introduced e.g. in Chapter~V of \cite{HP98}. For
convenience, and shortening the computations, we introduce the ${\mathcal P}$ notation: We say $x$ is of ${\mathcal P}(y)$, when  the code of $x$ is bounded  above by a polynomial of $y$; and we write this as $\ulcorner x\urcorner\leqslant{\mathcal P}(y)$, meaning that for some $n$ the inequality $\ulcorner x\urcorner\leqslant y^n+n$ holds. Let us note that $X\leqslant{\mathcal P}(Y)$ is equivalent to the old (more familiar) $O-$notation $``\log  X\!
\in\!{\mathcal O}(\log Y)"$. Here we collect some very basic facts about this fixed efficient coding  that will be needed later.

\begin{remark} {\rm Let $A,B$ be  sets of terms and let $|A|,|B|$ denote their cardinality. Then

$\bullet \ \ulcorner A\cup B\urcorner
\leqslant 64\cdot(\ulcorner A\urcorner\cdot \ulcorner B\urcorner)$ (Proposition 3.29 page 311  of \cite{HP98}); and

$\bullet \ \left(|A|\right)\leqslant (\log \ \ulcorner A\urcorner)$ (
 Section (e) pages 304--310  of \cite{HP98});
 \hfill $\subset\!\!\!\!\supset$
}\end{remark}
Let $\mathcal{L}_A=\langle0,{\mathfrak s},+,\cdot,\leqslant\rangle$ be the language of arithmetics (see Example \ref{example2}). If we let ${\mathcal L}_A^{\rm Sk}$ to be the closure of ${\mathcal L}_A$ under Skolem function and constant symbols, i.e., let ${\mathcal L}_A^{\rm Sk}$  be the smallest set that contains ${\mathcal L}_A$ and for any ${\mathcal L}_A^{\rm Sk}-$formula $\exists x \phi(x)$ we have ${\mathfrak f}_{\exists x \phi(x)}\!\in\!{\mathcal L}_A^{\rm Sk}$, then
  this new countable language can also be re-coded, and this recoding can be generalized to ${\mathcal L}_A^{\rm Sk}-$terms and ${\mathcal L}_A^{\rm Sk}-$formulas.
We wish to compute an upper bound for the codes of evaluations on a
set of terms $\Lambda$. For a given $\Lambda$, all the atomic
formulas, in the language ${\mathcal L}_A$, constructed from
terms of $\Lambda$ are either of the form $t=s$ or of the form
$t\leqslant s$ for some $t,s\!\in\!\Lambda$. And every member of an
evaluation $p$ on $\Lambda$ is an ordered pair like $\langle
t=s,i\rangle$ or $\langle t\leqslant s,i\rangle$ for some $t,s\!\in\!\Lambda$
and $i\!\in\!\{0,1\}$. Thus the code of any member of $p$ is a constant multiple of
$(\ulcorner t\urcorner\cdot \ulcorner s\urcorner)^2$, and so the code of $p$ is bounded above by
${\mathcal P}(\prod_{t,s\in\Lambda}\ulcorner t\urcorner\cdot \ulcorner s\urcorner)$.
Let us also note that $\prod_{t,s\in\Lambda}\ulcorner t\urcorner\cdot
\ulcorner s\urcorner=\prod_{t\in\Lambda}(\ulcorner t\urcorner)^{2|\Lambda|}=
(\prod_{t\in\Lambda}\ulcorner t\urcorner)^{2|\Lambda|}\leqslant{\mathcal P}(\ulcorner\Lambda\urcorner)
^{2\log\ulcorner\Lambda\urcorner}\leqslant{\mathcal P}(\ulcorner\Lambda\urcorner^{\log\ulcorner\Lambda\urcorner})\leqslant
{\mathcal P}(\omega_1(\ulcorner\Lambda\urcorner))$. Thus we have $\ulcorner p\urcorner \leqslant{\mathcal P}\left(\omega_1(\ulcorner\Lambda\urcorner)\right)$ for any evaluation $p$ on any set of terms $\Lambda$. As noted in \cite{Sal10} there are $\exp(2|\Lambda|^2)$ different evaluations on the set $\Lambda$, and by $|\Lambda|\leqslant \log\ulcorner\Lambda\urcorner$ we get  $\exp(2|\Lambda|^2)\leqslant {\mathcal P}\left(\exp((\log\ulcorner\Lambda\urcorner)^2)\right)
\leqslant{\mathcal P}\left(\omega_1(\ulcorner\Lambda\urcorner)\right)$. So, only when  $\omega_1(\ulcorner\Lambda\urcorner)$ exists, can we have all the evaluations on $\Lambda$ in our disposal.
We need an  upper bound on the size (cardinal) and the code of $\Lambda^{\langle j\rangle}$ defined above.
\begin{theorem}\label{log4}
{\rm If for a set of terms $\Lambda$ with non-standard
$\ulcorner\Lambda\urcorner$ the value
$\omega_2(\ulcorner\Lambda\urcorner)$ exists, then for a non-standard
$j$ the value $\ulcorner\Lambda^{\langle j\rangle}\urcorner$ will
exist.
}\end{theorem}

\begin{proof}
We first show that the following inequalities hold when $\ulcorner\Lambda\urcorner$ and
$|\Lambda|$ are sufficiently larger than $n$:
 (1)  $|\Lambda^{\langle n\rangle}|\leqslant{\mathcal P}\left( |\Lambda|^{n!}\right)$, and
 (2)   $\ulcorner\Lambda^{\langle n\rangle}\urcorner\leqslant
{\mathcal P}\Big(\big(\ulcorner\Lambda\urcorner\big)^{|\Lambda|^{(n+1)!}}\Big)$.

 \noindent Denote $\ulcorner\Lambda^{\langle k\rangle}\urcorner$ by $\lambda_k$ (thus $\ulcorner\Lambda\urcorner=\lambda_0=\lambda$) and $|\Lambda^{\langle k\rangle}|$ by $\sigma_k$ (and thus $|\Lambda|=\sigma_0=\sigma$). We first note that $\sigma_{k+1}\leqslant \sigma_k+M\sigma_k^M+k\sigma_k^k$ for a fixed $M$. Thus $\sigma_{k+1}\leqslant{\mathcal P}(\sigma_k^{k+1})$, and then, by an inductive argument, we have $\sigma_n\leqslant{\mathcal P}(\sigma^{n!})$. For the second statement, we first compute an upper bound for the code of the Cartesian power $A^m$ for a set $A$. Now we have $\ulcorner A^{k+1}\urcorner\leqslant {\mathcal P}\big(\prod_{t\in A^k  \&  s\in A}\ulcorner t\urcorner \cdot \ulcorner s\urcorner\big)\leqslant {\mathcal P}\big( \ulcorner A^k\urcorner^{|A|}\cdot\ulcorner A\urcorner^{|A|^k}\big)$, and thus $\ulcorner A^m\urcorner \leqslant {\mathcal P}\big(\ulcorner A\urcorner ^{|A|^m}\big)$ can be shown by induction on $m$. Now we have $\lambda_{k+1}\leqslant{\mathcal P}\big(\ulcorner \Lambda^{\langle k\rangle}\urcorner\cdot\ulcorner (\Lambda^{\langle k\rangle})^M\urcorner\cdot\ulcorner (\Lambda^{\langle k\rangle})^k\urcorner\big)$ for a fixed $M$. So, $\lambda_{k+1}\leqslant {\mathcal P}\big(\lambda_k^{{\sigma_k}^{k}}\big)$ and finally our desired conclusion $\lambda_m\leqslant{\mathcal P}\big(\lambda^{\sigma^{(m+1)!}}\big)$ follows by induction.

Now since $\ulcorner\Lambda\urcorner$ is a non-standard number, there must exist a non-standard $j$ such that $j\!\leqslant\!\log^4(
\ulcorner\Lambda\urcorner)$. Thus $2(j+1)!\leqslant
\exp^2(j)\!\leqslant\!\log^2(\ulcorner\Lambda\urcorner)$. Now, by  the inequality (2) above  we can write

$\ulcorner\Lambda^{\langle
j\rangle}\urcorner\!\leqslant\!{\mathcal P}\left((\ulcorner\Lambda\urcorner)^{|\Lambda|^{(j+1)!}}\right)
\!\leqslant\!{\mathcal P}
\left((2^{2\log\ulcorner\Lambda\urcorner})^{(\log\ulcorner\Lambda\urcorner)^{
(j+1)!}}\right)\!\leqslant\!{\mathcal P}\left(\exp((\log\ulcorner\Lambda\urcorner)^{2(j+1)!})\right)\!\leqslant\!{\mathcal P}\left(\exp(\omega_1(\log\ulcorner\Lambda\urcorner))\right)$, and so
$\ulcorner\Lambda^{\langle
j\rangle}\urcorner\leqslant{\mathcal P}\left(\omega_2(\ulcorner\Lambda\urcorner)\right)$.
\end{proof}

The reason that Theorem \ref{log4} is stated for non-standard
$\ulcorner\Lambda\urcorner$ is that the set $\Lambda^{\langle\infty\rangle}$, needed
for constructing the model ${\textswab M}(\Lambda,p)$,  is not
definable in ${\mathcal L}_A$. But the existence of the definable
$\Lambda^{\langle j\rangle}$ for a non-standard $j$ can guarantee
the existence of $\Lambda^{\langle\infty\rangle}$ and thus of
${\textswab M}(\Lambda,p)$. This non-standard $j$ exists for
non-standard $\ulcorner\Lambda\urcorner$.
The above Theorem \ref{log4} suggest the following formalization of the notion of Herbrand Consistency:
\begin{definition}
{\rm A theory $T$ is called {\em Herbrand Consistent} if for any set of terms $\Lambda$ (constructed from the Skolem terms of $T$) for which $\omega_2(\ulcorner\Lambda\urcorner)$ exists, there is a $T-$evaluation on $\Lambda$.

\noindent This notion can be formalized in the language of arithmetic, denoted by ${\rm HCon}(T)$. \hfill $\subset\!\!\!\!\supset$
}\end{definition}

\section{Separating Bounded Arithmetical Hierarchy}
\subsection{Separating by Herbrand Consistency}
Let us recall that the (usual) Hilbert Provability $T\vdash\varphi$ is, by definition, the existence of a sequence of formulas whose last element is (the G\"odel code of) $\varphi$ and every other element is either a logical axiom or an axiom of $T$, or has been resulted from two previous elements by means of model ponens.  Thus Hilbert Consistency means the non-existence of  such a sequence whose last element is a contradiction. Let us note that Herbrand Consistency is, in a sense,  a weaker notion of consistency; some more explanation is in order.  The super-exponentiation  function is  defined by  ${\rm sup_-exp}(x)=\exp^x(x)$; let ${\rm Sup_-Exp}$ be the sentence which expresses the totality of this function (${\rm Sup_-Exp} = \forall x\exists y [y={\rm sup_-exp}(x)]$).  By the techniques of cut elimination (see e.g. \cite{HP98}) it can be shown that ${\rm I\Delta_0+Sup_-Exp}\vdash {\rm Con}(T)\leftrightarrow{\rm HCon}(T)$ for any theory $T$. Though the theory ${\rm I\Delta_0+Exp}$ it too weak to recognize this equivalence, since ${\rm I\Delta_0+Exp}\vdash{\rm HCon}({\rm Q})$  but ${\rm I\Delta_0+Exp}\not\vdash{\rm Con}({\rm Q})$ (\cite{HP98}, Theorem~5.20 and Corollary~5.29). So,  ${\rm I\Delta_0+Exp}\not\vdash {\rm HCon}(T)\rightarrow{\rm Con}(T)$ in general, though it can be shown  that ${\rm I\Delta_0+Exp}\vdash {\rm Con}(T)\rightarrow{\rm HCon}(T)$ (see \cite{HP98}).
Thus showing the unprovability of Herbrand Consistency of weak theories in themselves is an interesting generalization of G\"odel's Second Incompleteness Theorem. What we are interested in here, is whether the notion of Herbrand Consistency can $\Pi_1-$separate the  hierarchy $\circledast$ above. We already know (only) that ${\rm I\Delta_0+Exp}$ is not $\Pi_1-$conservative over ${\rm I\Delta_0+\bigwedge\Omega_j}$, but we do not yet know whether  ${\rm I\Delta_0+Exp}$ is able to derive the Herbrand Consistency of the theory ${\rm I\Delta_0+\bigwedge\Omega_j}$ or not.
\begin{conjecture}
{\rm The notion of Herbrand Consistency cannot $\Pi_1-$separate the (already $\Pi_1-$distinct) theories ${\rm I\Delta_0+Exp}$ and ${\rm I\Delta_0+\bigwedge\Omega_j}$; that is  ${\rm I\Delta_0+Exp}\not\vdash{\rm HCon}({\rm I\Delta_0+\bigwedge\Omega_j})$.\hfill $\subset\!\!\!\!\supset$
}\end{conjecture}
Though, for any $m\!\!\geqslant\!\!1$, Herbrand Consistency can $\Pi_1-$separate ${\rm I\Delta_0+Exp}$ from the theory ${\rm I\Delta_0+\Omega_m}$, and also from  ${\rm I\Delta_0}$. Since  already  ${\rm I\Delta_0+\Omega_m}\not\vdash{\rm HCon}({\rm I\Delta_0+\Omega_m})$ for any $m\!\!\geqslant\!\!1$ (see \cite{Ada02,Sal02}) and also the following theorem hold.
\begin{theorem}\label{exp}
{\rm For any $m\geqslant 1$ we have ${\rm I\Delta_0+Exp}\vdash{\rm HCon}({\rm I\Delta_0+\Omega_m})$.
}\end{theorem}

\begin{proof} Reason inside a model ${\cal M}\models{\rm I\Delta_0+Exp}$. For any set of terms $\Lambda\!\in\!{\cal M}$, assume it has been rearranged in a non-decreasing order $\Lambda=\{t_0, t_1, t_2, \cdots, t_j\}$.  Then for some $u_1, u_2, \cdots, u_j$ we have the inequalities $t_1\!\leqslant\!\omega_m^{u_1}(t_0), t_2\!\leqslant\!\omega_m^{u_2}(t_1), \cdots, t_j\!\leqslant\!\omega_m^{u_j}(t_{j-1})$. Let $u=\sum_iu_i$; then $t_i\!\leqslant\!\omega_m^{u}(t_0)$ for each $i\!\leqslant\!j$. On the other hand, $\omega_m^u(t_0)=\exp^m([\log^m(t_0)]^{\exp(u)})\!\leqslant\!\exp^{m+1}(u\cdot t_0)$; and since $u\!\leqslant\!(\ulcorner\Lambda\urcorner)^2$ and $\exp$ is available for all elements, then every term in $\Lambda$ has a realization inside ${\cal M}$. Denote the realization of $t_i$ by $t^{\cal M}_i$. Then the evaluation $p$ defined on $\Lambda$ by the putting

{$(1) \ p\models t_k=t_l$ if and only if $t_k^{\cal M}=t_l^{\cal M}$, and $(2) \ p\models t_k\!\leqslant\!t_l$ if and only if $t_k^{\cal M}\!\leqslant\!t_l^{\cal M}$,}

\noindent is an $({\rm I\Delta_0+\Omega_m})-$evaluation on $\Lambda$ (note that ${\cal M}\models{\rm I\Delta_0+\Omega_m}$).
\end{proof}
\begin{remark}
{\rm By the above proof it can also be shown that ${\rm I\Delta_0+Exp}\vdash{\rm HCon}({\rm I\Delta_0})$ and it is shown in \cite{Sal10} that ${\rm I\Delta_0}\not\vdash{\rm HCon}({\rm I\Delta_0})$. Thus ${\rm HCon}(-)$ can $\Pi_1-$separate ${\rm I\Delta_0+Exp}$ and ${\rm I\Delta_0}$ as well. \hfill $\subset\!\!\!\!\supset$
}\end{remark}
\begin{remark}
{\rm A reason that the proof of the above theorem cannot be applied for showing the presumably false deduction ${\rm I\Delta_0+Exp}\vdash{\rm HCon}({\rm I\Delta_0+\bigwedge\Omega_j})$ in the conjecture, is that for the set of terms $\Xi=\{v_0, v_1, \cdots, v_j\}$ defined by $v_0\!=\!4$ and $v_{i+1}\!=\!\omega_{i+1}(v_i)$ for each $i\!<\!j$, we have $v_j=\exp^j(4)$ (the equality $v_i=\exp^i(4)$ follows by induction on $i$). Thus  a model of ${\rm I\Delta_0+Exp}$ can contain a big $j$, and the set $\Xi$ above, for which $\exp^j(4)$ does not exist. So, some terms of $\Xi$ may not have a realization in the model; and a suitable evaluation could not be defined in it. Note that $\exp^j(4)$ is a super-exponential term and cannot be obtained by applying a finite number of  the exponential function.
\hfill $\subset\!\!\!\!\supset$
}\end{remark}
\subsection{Unprovability of Herbrand Consistency of ${\rm I\Delta_0}$ in ${\rm I\Delta_0+\bigwedge\Omega_j}$}
Here we show the unprovability of the Herbrand Consistency of ${\rm I\Delta_0}$ in ${\rm I\Delta_0+\bigwedge\Omega_j}$. The proof is by a technique of logarithmic shortening of bounded witnesses, introduced by Z. Adamowicz in \cite{Ada02}, and also employed in \cite{Kol06,Sal10}. The following is an outline of the proof. If ${\rm I\Delta_0+\bigwedge\Omega_j}\vdash{\rm HCon}({\rm I\Delta_0})$, then there is an $\textswab{m}\!\geqslant\!2$ such that
 \  {($\pitchfork$)  \ ${\rm I\Delta_0+\Omega_\textswab{m}}\vdash{\rm HCon}({\rm I\Delta_0})$.}
From now on fix this $\textswab{m}$.We first show that one cannot always logarithmically shorten the witness of a bounded formula inside ${\rm I\Delta_0+\Omega_\textswab{m}}$. Or in other words, for any cut (i.e., a definable initial segment) like $I$ and its logarithme $J=\{\log x\!\!\mid\!\!x\!\in\!I\}$, there exists a bounded formula $\eta(x)$ such that the theory $({\rm I\Delta_0+\Omega_\textswab{m}})+\exists x\!\!\in\!\!I\eta(x)$ is consistent, but the theory $({\rm I\Delta_0+\Omega_\textswab{m}})+\exists x\!\!\in\!\!J\eta(x)$ is not consistent; or in other words we have  ${\rm I\Delta_0+\Omega_\textswab{m}}\vdash\forall x\!\!\in\!\!J\neg\eta(x)$ and ${\rm I\Delta_0+\Omega_\textswab{m}}\not\vdash\forall x\!\!\in\!\!I\neg\eta(x)$. For a similar statement on ${\rm I\Delta_0+\Omega_1}$ see Theorem 5.36 of \cite{HP98}. Second
we show that, under the  assumption ($\pitchfork$) above, for any bounded  $\theta(x)$, if the theory $({\rm I\Delta_0+\Omega_\textswab{m}})+\exists x\!\!\in\!\!I\theta(x)$ is consistent, then so is  $({\rm I\Delta_0+\Omega_\textswab{m}})+\exists x\!\!\in\!\!J\theta(x)$. This immediately contradicts ($\pitchfork$). The first theorem is a classical result in the theory of bounded arithmetic, whic can be proved without using the assumption ($\pitchfork$). The second theorem uses the assumption ($\pitchfork$) to be able to logarithmically shorten a witness $a\!\in\!I\wedge\theta(a)$ for the formula $x\!\in\!I\wedge\theta(x)$ in a model ${\cal M}\models({\rm I\Delta_0+\Omega_\textswab{m}})+\exists x\!\!\in\!\!I\theta(x)$ by constructing a model ${\cal N}\models({\rm I\Delta_0+\Omega_\textswab{m}})+\exists x\!\!\in\!\!J\theta(x)$. And for that we will use the assumption ($\pitchfork$) to infer ${\cal M}\models{\rm HCon}({\rm I\Delta_0})$, which implies the existence of an ${\rm I\Delta_0}-$evaluation on any set of terms $\Lambda$ for which $\omega_2(\ulcorner\Lambda\urcorner)$ exists. That evaluation on a suitable $\Lambda$ will give us a model of ${\rm I\Delta_0}+\exists x\!\!\in\!\!J\theta(x)$ (see Lemma \ref{lm:mlambdap}). Then by a trick of \cite{Kol06} we will construct a model for the theory $({\rm I\Delta_0+\Omega_\textswab{m}})+\exists x\!\!\in\!\!J\theta(x)$. The suitable set of terms $\Lambda$ should contain a term for representing $a$ and all the polynomials (i.e., arithmetical terms) of $a$. Define the terms $\underline{j}$'s by induction: $\underline{0}=0$, and  $\underline{j+1}={\mathfrak s}(\underline{j})$. The term $\underline{j}$ represents the (standard or non-standard) number $j$. We require that $\Lambda\supseteq\{\underline{j}\!\mid\!j\!\leqslant\!\omega_1(a)\}=\digamma$. The code of $\digamma$ is bounded above by $\ulcorner\digamma\urcorner\leqslant{\mathcal P}\left(\prod_{j=0}^{j=\omega_1(\alpha)}2^j\right)
\!\leqslant\!{\mathcal P}\left(\exp(\omega_1(\alpha)^2)\right)$. And the value $\omega_2(\ulcorner\digamma\urcorner)$ is bounded above by $\omega_2(\ulcorner\digamma\urcorner)\!\leqslant\!{\mathcal P}\left(\omega_2(\exp(\omega_1(\alpha)^2))\right)
\!\leqslant\!{\mathcal P}\left(\exp(\omega_1(\omega_1(\alpha)^2))\right)
\!\leqslant\!{\mathcal P}\left(\exp^2\left(4(\log\alpha)^4\right)\right)$.
\begin{definition}
{\rm Let the cut ${\mathcal I}$ be defined by ${\mathcal I}=\{x\mid \exists y[y=\exp^2\big(4(\log\alpha)^4\big)]\}$ and its logarithm be ${\mathcal J}=\{x\mid \exists y [y=\exp^2\big(4\alpha^4\big)]\}$.
}\end{definition}
Note that $\forall x  [  \exp(x)\!\in\!{\mathcal I}\iff x\!\in\!{\mathcal J}  ]$. The two mentioned theorems are the following.
\begin{theorem}\label{th:a}
{\rm There exists a bounded formula $\eta(x)$ such that the theory $({\rm I\Delta_0+\Omega_\textswab{m}})+\exists x\!\!\in\!\!\mathcal{I}\eta(x)$ is consistent, but the theory $({\rm I\Delta_0+\Omega_\textswab{m}})+\exists x\!\!\in\!\!\mathcal{J}\eta(x)$ is not consistent
}\end{theorem}
\begin{theorem}\label{th:b}
{\rm If ${\rm I\Delta_0+\Omega_\textswab{m}}\vdash{\rm HCon}({\rm I\Delta_0})$, then for any bounded formula $\theta(x)$, the consistency of the theory $({\rm I\Delta_0+\Omega_\textswab{m}})+\exists x\!\!\in\!\!\mathcal{I}\theta(x)$ implies the consistency of  $({\rm I\Delta_0+\Omega_\textswab{m}})+\exists x\!\!\in\!\!\mathcal{J}\theta(x)$.
}\end{theorem}
Having proved the theorems below, we conclude our main result.
\begin{corollary}\label{main}
{\rm For any ${m}\!\in\!\mathbb{N}$, ${\rm I\Delta_0+\Omega_{m}}\not\vdash{\rm HCon}({\rm I\Delta_0})$; thus ${\rm I\Delta_0+\bigwedge\Omega_j}\not\vdash{\rm HCon}({\rm I\Delta_0})$.
}\end{corollary}
We can alreay prove Theorem \ref{th:a}, which is an interesting theorem in its own right.

\begin{proof}  ({\bf of Theorem \ref{th:a}}.)  The proof is rather long and we will sketch the main ideas, cf. the proof of Theorem 5.36 in \cite{HP98}. We will follow \cite{Ada02} here. If the theorem does not hold, then for {\em any} bounded formula $\theta(x)$, the consistency of  $({\rm I\Delta_0+\Omega_\textswab{m}})+\exists x\!\!\in\!\!\mathcal{I}\theta(x)$ will imply the consistency of  $({\rm I\Delta_0+\Omega_\textswab{m}})+\exists x\!\!\in\!\!\mathcal{J}\theta(x)$. Now let $\psi(x)$ be a bounded formula such that the theory $({\rm I\Delta_0+\Omega_\textswab{m}})+\exists x\!\!\in\!\!\mathcal{I}\psi(x)$ is consistent. Then $({\rm I\Delta_0+\Omega_\textswab{m}})+\exists x\!\!\in\!\!\mathcal{J}\psi(x)$ is consistent also. The formula $\exists x\!\!\in\!\!\mathcal{J}\psi(x)$ is equivalent to $\exists y\!\!\in\!\!\mathcal{I}\psi'(y)$ where $\psi'(y)=\exists x\!\!\leqslant\!\!y(y\!=\!\exp(x)\!\wedge\!\psi(x))$ is a bounded formula. So, the theory $({\rm I\Delta_0+\Omega_\textswab{m}})+\exists y\!\!\in\!\!\mathcal{I}\psi'(y)$ is consistent, and by the assumption,  the theory $({\rm I\Delta_0+\Omega_\textswab{m}})+\exists y\!\!\in\!\!\mathcal{J}\psi'(y)$ must be consistent too. Again the formula $\exists y\!\!\in\!\!\mathcal{J}\psi'(y)$ is equivalent to $\exists z\!\!\in\!\!\mathcal{I}\exists x\!\!\leqslant\!\!z(z\!=\!\exp^2(x)\!\wedge\!\psi(x))$. Continuing this way, we infer that the theory $({\rm I\Delta_0+\Omega_\textswab{m}})+\exists u\!\!\in\!\!\mathcal{I}\exists x\!\!\leqslant\!\!u(u\!=\!\exp^k(x)\!\wedge\!\psi(x))$ is consistent for any natural $k\!\in\!\mathbb{N}$. Let $\mathfrak{b}$ be a constant symbol. By the above argument, the theory $({\rm I\Delta_0+\Omega_\textswab{m}})+\{\exists z[z=\exp^k(\mathfrak{b})\!\wedge\!\psi(\mathfrak{b})]\!\mid\!k\!\in\!\mathbb{N}\}$ is finitely consistent, and whence it is consistent. Thus there exists a model $\mathcal{K}\models{\rm I\Delta_0}$ such that for some element $b\!\in\!\mathcal{K}$, $\mathcal{K}\models\exists z[z=\exp^k(b)\!\wedge\!\psi(b)]$ for any $k\!\in\!\mathbb{N}$. The initial segment $\mathcal{M}$ of $\mathcal{K}$ determined by   $\{a\!\in\!\mathcal{K}\!\mid\!\exists k\!\in\!\mathbb{N}: a\!\leqslant\!\exp^k(b)\}=\exp^{\mathbb{N}}(b)$ is a model of ${\rm I\Delta_0+Exp}$ for which $\mathcal{M}\models\psi(b)$. Thus the theory $({\rm I\Delta_0+Exp})+\exists x\psi(x)$ is consistent. Hence, if the theorem is not true, then for {\em any} bounded formula $\psi(x)$, if the theory $({\rm I\Delta_0+\Omega_\textswab{m}})+\exists x\!\!\in\!\!\mathcal{I}\theta(x)$ is consistent, then $({\rm I\Delta_0+Exp})+\exists x\psi(x)$ is also consistent. Contrapositing this statement, we get: if for a $\Pi_1-$formula $\forall x\theta(x)$ (with bounded $\theta$) we have ${\rm I\Delta_0+Exp}\vdash\forall x\theta(x)$, then we must also have ${\rm I\Delta_0+\Omega_\textswab{m}}\vdash\forall x\!\!\in\!\!\mathcal{I}\theta(x)$. Since for any $x\!\!\in\!\!\mathcal{I}$ the value $\exp^3(x)$ exists, and all the finite applications of $\omega_{\textswab{m}}$ are dominated by one use of $\exp$, then ${\rm I\Delta_0+\Omega_\textswab{m}}\vdash\forall x\!\!\in\!\!\mathcal{I}\theta(x)$ implies that ${\rm I\Delta_0}\vdash\forall x[\exists y(y=\exp^4(x))\rightarrow\theta(x)]$. All in all, from the falsity of the theorem we inferred that whenever ${\rm I\Delta_0+Exp}\vdash\forall x\theta(x)$ for a bonded $\theta(x)$, then ${\rm I\Delta_0}\vdash\forall x[\exists y(y=\exp^4(x))\rightarrow\theta(x)]$. Or in other words, four times application of ${\rm Exp}$ is engough to deduce all the $\Pi_1-$theorems of ${\rm I\Delta_0+Exp}$! And this is in contradiction with Theorem 5.36 of \cite{HP98}.
\end{proof}

The rest of the paper will be dedicated to proving Theorem \ref{th:b}.
\begin{definition}
{\rm The inverse of $\omega_n$, denoted by  $\varpi_n(x)$, is defined to be the smallest $y$ such that the inequality $\omega_n(y)\!\geqslant\!x$ holds.
The cut $\mathfrak{I}_n$ is the set $\{x\!\mid\!\exists y[y=\exp^2(\varpi_{n-1}(4x^4))]\}$.
\hfill $\subset\!\!\!\!\supset$
}\end{definition}
Let us note that $\mathcal{J}\!\subset\!\mathfrak{I}_n\!\subset\!\mathcal{I}$ holds for any $n\!>\!1$. We prove Theorem \ref{th:b}  by  an auxiliary  theorem.
\begin{theorem}\label{th:c}
{\rm If ${\rm I\Delta_0+\Omega_\textswab{m}}\vdash{\rm HCon}({\rm I\Delta_0})$, then for any bounded formula $\theta(x)$, the consistency of the theory $({\rm I\Delta_0+\Omega_\textswab{m}})+\exists x\!\in\!\mathcal{I}\theta(x)$ implies the consistency of  $({\rm I\Delta_0+\Omega_\textswab{m}})+\exists x\!\in\!\mathfrak{I}_\textswab{m}\theta(x)$.
}\end{theorem}
Having proved this, Theorem \ref{th:b} can be proved easily:

\begin{proof} ({\bf of Theorem \ref{th:b} from Theorem \ref{th:c}}.)
Assume ${\rm I\Delta_0+\Omega_\textswab{m}}\vdash{\rm HCon}({\rm I\Delta_0})$. Let $\theta(x)$ be a bounded formula such that $({\rm I\Delta_0+\Omega_\textswab{m}})+\exists x\!\in\!\mathcal{I}\theta(x)$ is consistent. Then by Theorem \ref{th:c},    $({\rm I\Delta_0\!+\!\Omega_\textswab{m}})\!+\!\exists x\!\in\!\mathfrak{I}_\textswab{m}\theta(x)$ is consistent too. Let $\theta'(y)$ be the bounded formula $\theta'(y)=\exists x\!\leqslant\!y[4x^4\!\leqslant\!\omega_{\textswab{m}-1}(4(\log y)^4)\!\wedge\!\theta(x)]$; then $\exists x\!\in\!\mathfrak{I}_\textswab{m}\theta(x)$ is equivalent to $\exists y\!\in\!\mathcal{I}\theta'(y)$. Now, since the theory $({\rm I\Delta_0+\Omega_\textswab{m}})+\exists y\!\in\!\mathcal{I}\theta'(y)$ is consistent, again by Theorem \ref{th:c}, the theory $({\rm I\Delta_0+\Omega_\textswab{m}})+\exists y\!\in\!\mathfrak{I}_\textswab{m}\theta'(y)$ must be consistent. Then we note that the implication  $(y\!\in\!\mathfrak{I}_\textswab{m})\!\wedge\![4x^4\!\leqslant\!\omega_{\textswab{m}-1}(4(\log y)^4)]\Rightarrow (x\!\in\!\mathcal{J})$ holds for non-standard $x$ and $y$, because $\omega_{\textswab{m}-1}^2(4[\log y]^4)\!<\!4y^4$. So, $({\rm I\Delta_0+\Omega_\textswab{m}})+\exists x\!\in\!\mathcal{J}\theta(x)$ is consistent.
\end{proof}

For proving Theorem \ref{th:c} we assume that
 for the bounded formula $\theta(x)$ there exists a model $\mathcal{M}$ such that
{$(\circ)   \ \mathcal{M}\models({\rm I\Delta_0+\Omega_\textswab{m}})\!+\!(\alpha\!\in\!\mathcal{I}\!\wedge\!\theta(\alpha))$}
holds for some non-standard $\alpha\!\in\!\mathcal{M}$.
 We will construct another model 
{$\mathcal{N}\models({\rm I\Delta_0+\Omega_\textswab{m}})+\exists x\!\!\in\!\!\mathfrak{I}_\textswab{m}\theta(x).$}
Define the terms $\underline{j}$'s by induction: $\underline{0}=0$, and  $\underline{j+1}={\mathfrak s}(\underline{j})$. The term $\underline{j}$ represents the (standard or non-standard) number $j$. Let ${\mathfrak q}$ be the Skolem function symbol for the formula ${\exists y(y\leqslant x^2\wedge y=x^2)}$ and  ${\mathfrak c}$ be the Skolem constant symbol for the  sentence ${\exists x\big(\exists w (w\leqslant x^2\wedge w=x^2)\wedge\forall v (v\not\leqslant({\mathfrak s} x)^2\wedge v\not=({\mathfrak s} x)^2)\big)}$, and let $\Upsilon$ be the following set of term $\Upsilon=\{0,  0+0,  0^2, {\mathfrak c}, {\mathfrak c}^2, {\mathfrak c}^2+0, {\mathfrak s}{\mathfrak c}, {\mathfrak q}{\mathfrak c},  ({\mathfrak s}{\mathfrak c})^2,  ({\mathfrak s}{\mathfrak c})^2+0\}$ (see Example \ref{q}). Define the terms ${\sf z}_i$'s inductively: ${\sf z}_0=\underline{2}$, and ${\sf z}_{j+1}={\mathfrak q}({\sf z}_j)$. Since we will have ${\sf q}(x)=x^2$ in ${\rm I\Delta_0}-$evaluations (by Example \ref{q}), then ${\sf z}_i$ will represent $\exp^2(i)$ (can be verified by induction on $i$).  Take {$\Lambda=\Upsilon\cup\{\underline{j}\!\mid\!j\!\leqslant\!\omega_1(\alpha)\}\cup\{{\sf z}_j\!\mid\!j\!\leqslant\!4\alpha^4\}$}; then
$\omega_2(\ulcorner\Lambda\urcorner)$ is of order  $\exp^2\big(4(\log\alpha)^4\big)$ which exists  by the assumption ${\mathcal M}\models \alpha\!\in\!{\mathcal I}$ (see $(\circ)$ above). Since by the assumptions $(\pitchfork)$ and $(\circ)$ we have $\mathcal{M}\models{\rm HCon}({\rm I\Delta_0})$, then there must exist an ${\rm I\Delta_0}-$evaluation $p\!\in\!\mathcal{M}$. Now, we can build the model $\mathcal{K}:={\textswab M}(\Lambda,p)$.
\begin{lemma}\label{theta}
{\rm With the above assumptions, $\mathcal{K}\models\theta(\underline{\alpha}/p)$.
}\end{lemma}
After proving this lemma, we can finish the proof of Theorem \ref{th:c}.

\begin{proof} ({\bf of Theorem \ref{th:c} from Lemma \ref{theta}}.)
  By Lemma  \ref{lm:mlambdap} we already have $\mathcal{K}\models{\rm I\Delta_0}$, and by Lemma~\ref{theta}, $\mathcal{K}\models\theta(\underline{\alpha}/p)$. We can  see that $\underline{\alpha}/p\!\in\!\mathcal{J}^{\mathcal{K}}$ by the existence of ${\sf z}_i/p$'s ($\mathcal{K}\models{\sf z}_{4\alpha^4}/p=\exp^2(4[\underline{\alpha}/p]^2)$). Whence $\mathcal{K}\models\alpha/p\!\in\!\mathcal{J}\!\wedge\!\theta(\underline{\alpha}/p)$. By Lemma \ref{lem2} there exists some (non-standard) element $\beta\!\in\!\mathcal{K}$ such that  the inequalities $\omega_\textswab{m}^\mathbb{N}(\beta)<{\sf z}_{4\alpha^4}/p\!\leqslant\!\omega_{\textswab{m}+1}(\beta)$ hold.
Now, let $\mathcal{N}$ be the initial segment of $\mathcal{K}$ determined by $\omega_\textswab{m}^\mathbb{N}(\beta)$, i.e., $\mathcal{N}\!=\!\{x\!\in\!\mathcal{K}\!\mid\!\exists k\!\in\!\mathbb{N}:\! x\!<\!\omega_\textswab{m}^k(\beta)\}$. We have $\mathcal{N}\models({\rm I\Delta_0+\Omega_\textswab{m}})+\theta(\underline{\alpha}/p)$, and all we have to
o show is that $\mathcal{N}\models\underline{\alpha}/p\!\in\!\mathfrak{I}_\textswab{m}$.  First note that $\beta\!\in\!\mathcal{N}$, and second that $\exp^2(4[\underline{\alpha}/p]^4)\!\leqslant\!\omega_{\textswab{m}+1}(\beta)$ implies $4[\underline{\alpha}/p]^4\!\leqslant\!\omega_{\textswab{m}-1}(\log^2\beta)$, and so  we have $\varpi_{\textswab{m}-1}(4[\underline{\alpha}/p]^4)\!\leqslant\!\log^2\beta$. Thus  $\exp^2(\varpi_{\textswab{m}-1}(4[\underline{\alpha}/p]^4))$ exists ($\!\leqslant\!\beta$), and  so $[\underline{\alpha}/p]\!\in\!\mathfrak{I}_\textswab{m}$.
\end{proof}

Finally, it remains (only) to prove Lemma \ref{theta}. This is exactly Corollary 35 of \cite{Sal10}; and the reader is  invited to consult it for more details. Here we  sketch a proof, for the sake of self-containedness.

\begin{proof} ({\bf of Lemma \ref{theta} -- A Sketch}.)
Since $\theta(x)\!\in\!\Delta_0$ and $\mathcal{M}\models\theta(\alpha)$, we note that the range of the quantifiers of $\theta(\alpha)$ is the set $\{x\!\in\!\mathcal{M}\mid x\!\leqslant\!t(\alpha) \textrm{ \ for \ some \ }\mathcal{L}_A\textrm{--term} \ t\}$. This set is the initial segment of $\mathcal{M}$ determined by $\alpha^\mathbb{N}$; denote it by $\mathcal{M}'$. We have $\mathcal{M}'\models\theta(\alpha)$. For any $j\!\in\!\alpha^\mathbb{N}$ we have the corresponding $\underline{j}\!\in\!\Lambda$, and thus $\underline{j}/p\!\in\!\mathcal{K}$. So, this suggests a correspondence between ${\alpha}^\mathbb{N}$ and the initial segment of $\mathcal{K}$ determined by $(\underline{\alpha}/p)^\mathbb{N}$ which we denote it by $\mathcal{K}'$. It suffices to show that this correspondence exists and is an isomorphism between $\mathcal{M}'$ and $\mathcal{K}'$. Because, then we will have $\mathcal{K}'\models\theta(\underline{\alpha}/p)$ which will immediately imply $\mathcal{K}\models\theta(\underline{\alpha}/p)$ -- our desired conclusion.

We first note that $\mathcal{M}'=\{t(i_1,\ldots,i_n)\mid i_1,\ldots,i_n\!\leqslant\!\alpha \ \& \ t \ {\rm is \ an} \ \mathcal{L}_A\!\!-\!\!{\rm term}\}$. This follows from a more general fact: if for some model $\mathfrak{A}\models{\rm I\Delta_0}$ and  $x,a_1,\ldots,a_n\!\in\!\mathfrak{A}$ we have $\mathfrak{A}\models x\!\leqslant\!t(a_1,\ldots,a_n)$ for an $\mathcal{L}_A-$term $t$, then there are some $b_1,\ldots,b_m\!\in\!í\mathfrak{A}$ and some $\mathcal{L}_A-$term $s$ such that $\mathfrak{A}\models x\!=\!s(b_1,\ldots,b_m)$;  moreover $\max b_j\!\leqslant\!\max a_i$. This can be proved by induction on the complexity of $t$. For $t=t_1+t_2$, distinguish two cases: (1) if $\mathfrak{A}\models x\!\leqslant\!t_1(\overline{a})$, where $\overline{a}$ is a shorthand for $(a_1,\ldots,a_n)$, then we are done by the induction hypothesis; (2) if $\mathfrak{A}\models t_1(\overline{a})\!\leqslant\!x$ then there exists some $y\!\in\!\mathfrak{A}$ such that $\mathfrak{A}\models [x=t_1(\overline{a})+y]\!\wedge\![y\!\leqslant\!t_2(\overline{a})]$, and the result follows from the induction hypothesis.
For  $t=t_1\cdot t_2$, there are some $q,r\!\in\!\mathfrak{A}$ for which we have  $\mathfrak{A}\models [x=t_1(\overline{a})\!\cdot\!q+r]\!\wedge\![r\!<\!t_1(\overline{a})]\!\wedge\![q\!\leqslant\!t_2(\overline{a})]$. Two uses of induction hypothesis (for the terms $t_1$ and $t_2$) will finish the proof.
Similarly, 
$\mathcal{K}'=\{t(i_1,\ldots,i_n)\mid i_1,\ldots,i_n\!\leqslant\!\underline{\alpha}/p \ \& \ t \ {\rm is \ an} \ \mathcal{L}_A\!\!-\!\!{\rm term}\}$. Thus a correspondence by $t(i_1,\ldots,i_n)\mapsto t(\underline{i_1}/p,\ldots,\underline{i_n}/p)$ exists between $\mathcal{M}'$ and $\mathcal{K}'$.
That this mapping preserves atomic formulas of the form $u\!=\!v$ for  terms $u,v$ follows from the axioms of ${\rm Q}$. It also preserves atomic formulas of the form $u\!\leqslant\!v$ because we have in ${\rm Q}$ that $u\!\leqslant\!v\!\leftrightarrow\!w+u=v$ for some $w\!\leqslant\!v$. The preservation of negated atomic formulas follows from the ${\rm I\Delta_0}-$equivalences $x\!\not=\!y\!\leftrightarrow\!\mathfrak{s}y\!\leqslant\!x\!\vee\!\mathfrak{s}x\!\leqslant\!y$, and $x\!\not\leqslant\!y\!\leftrightarrow\!\mathfrak{s}y\!\leqslant\!x$. Thus the above  mapping is an isomorphism.
\end{proof}
\section{Conclusions}
We  saw one example of the provability of Herbrand Consistency of a theory $S$ in a (super-)theory (of it) $T$ (Theorem \ref{exp} for $S={\rm I\Delta_0+\Omega_m}, T={\rm I\Delta_0+Exp}$) and one example of the unprovability of Herbrand Consistency of $S$ in $T$ (Corollary \ref{main} for $S={\rm I\Delta_0}, T={\rm I\Delta_0+\bigwedge\Omega_j}$). The main point common in both of the results was that, if every Skolem term of $S$ has an evaluation in $T$, then $T$ may prove the Herbrand Consistency of $S$; but if there are some Skolem terms of $S$ which grow too fast for $T$ to catch them, then $T$ could not be able to derive the Herbrand Consistency of $S$. This is not a general law, but a rule of thumb.
Note that in our proof of Corollary \ref{main}, the terms ${\sf z}_i$ had the code of order $\exp(i)$ but the value of $\exp^2(i)$. And the theory ${\rm I\Delta_0+\bigwedge\Omega_j}$
 cannot catch the value of $\exp^2(i)$ by having the code $\exp(i)$; the gap is
 of exponential order. And in our proof of Theorem \ref{exp} the theory
 ${\rm I\Delta_0+Exp}$ could evaluate all the Skolem terms
 of ${\rm I\Delta_0+\Omega_m}$. A very similar argument can
 show that ${\rm I\Delta_0+Sup_-Exp}\vdash{\rm HCon}({\rm I\Delta_0+Exp})$.
 An open questione, asked by L. A. Ko{\l}odziejczyk, is that if showing the unprovability
 of Herbrand Consistency is possible without making use of fast-growing terms.
 More explicitly, if bounded formulas are required to have only variables in their bounds,
  and the re-axiomatization of ${\rm I\Delta_0}$ by the, rather restrictive, induction scheme $\forall y \big(\theta(0)\!\wedge\!\forall x\!<\!y[\theta(x)\!\rightarrow\!\theta({\mathfrak s}x)]\!\rightarrow\!\forall x\!\leqslant\!y\theta(x)\big)$ is taken into account, then is it possible to show the unprovability of the Herbrand Consistency of (this axiomatization of) ${\rm I\Delta_0}$ in itself? Note that here having terms like ${\sf z}_i$'s with double exponential values could not be possible.

The proof of our main result (Corollary \ref{main}) is very similar to the proof of the main result of \cite{Sal10} --  the unprovability ${\rm I\Delta_0}\not\vdash{\rm HCon}({\rm I\Delta_0})$. A major difference was the technique of Theorem \ref{th:c} for constructing a model of ${\rm I\Delta_0+\Omega_\textswab{m}}$ from a model of ${\rm I\Delta_0}$, for which Lemma \ref{lem2} was used. The idea of this technique is taken from \cite{Kol06}; note that the proof of our Theorem~\ref{th:a} is different from the proof of  the corresponding theorem in \cite{Kol06}, in that we had fixed one $\textswab{m}$ and followed the lines of the corresponding proof in \cite{Ada02}. That way we did not need to show the theorem for the theory ${\rm I\Delta_0+\bigwedge\Omega_j}$, and instead a simplified proof of the theorem for ${\rm I\Delta_0+\Omega_\textswab{m}}$ in \cite{Ada02} would suffice for us.

Let us finish the paper by repeating the open question asked also in \cite{Sal10}, which is whether G\"odel's Second Incompleteness Theorem for the Herbrand Consistency predicate has a uniform proof for theories containing Robinson's Arithmetic ${\rm Q}$.
\begin{question}{\rm
Can a {\sc Book} proof (in the words of Paul Erd\"{o}s) of
$T\not\vdash\mathcal{H}\textswab{C}\textswab{o}\textswab{n}(T)$ be given uniformly for any theory $T\supseteq {\rm Q}$ and a canonical  definition of Herbrand Consistency $\mathcal{H}\textswab{C}\textswab{o}\textswab{n}$?
}\end{question}
\subsection*{\underline{Acknowledgments}}
This work is a final result of a long project, pursued in different places and different times.
I would like to thank Professor Albert Visser for having me in Universiteit Utrecht as a visitor in April 2006, and Professor Gerhard J\"ager for inviting me to visit Universit\"at Bern in September 2006. I am grateful to Dr. Rashid Zaare-Nahandi in IASBS for his support during my stay there in the academic year 2006--2007.  I appreciate the stimulating discussions I had with Leszek Aleksander Ko{\l}odziejczyk during the LABCC workshop in Greifswald. This research was partially supported by the   National Elite Foundation of Iran
(\texttt{Bonyad Melli Nokhbegan} -- \url{http://www.bmn.ir/}).

\end{document}